\documentclass[11pt]{article}
\usepackage[centertags]{amsmath}
\usepackage{amsfonts}
\usepackage{amssymb}
\usepackage{amsthm}
\usepackage{newlfont}

\usepackage{enumerate}
\usepackage{color}

\newlength{\defbaselineskip}
\newcommand{\setlinespacing}[1]%
           {\setlength{\baselineskip}{#1 \defbaselineskip}}

\theoremstyle{plain}
\newtheorem{thm}{Theorem}
\newtheorem{cor}[thm]{Corollary}

\newtheorem{prop}[thm]{Proposition}
\theoremstyle{definition}
\newtheorem{defn}{Definition}
\theoremstyle{remark}
\newtheorem{rem}{\noindent Remark}
\numberwithin{equation}{section}

\theoremstyle{definition}

\renewcommand{\abstractname}{Abstract}

\begin{document}

\title{Conformal polynomial parameterizations}

\author{David P\'erez Fern\'andez\\
  Mathematics Department,\\
  Universidad Aut\'onoma de Madrid,\\
  \texttt{david.perez@uam.es}}

\date{May, 2012}

\maketitle

\renewcommand{\abstractname}{Abstract}
\begin{abstract}
The current paper discusses some new results about conformal polynomic surface parameterizations. A new theorem is proved: Given a conformal polynomic surface parameterization of any degree it must be harmonic on each component.\\

As a first geometrical application, every surface that admits a conformal polynomic parameterization must be a minimal surface. This is not the case for rational conformal polynomic parameterizations, where the conformal condition does not imply that components must be harmonic.\\

Finally, a new general theorem is established for conformal polynomic parameterizations of $m$-dimensional hypersurfaces, $m > 2$, in $\mathbb{R}^n$, with $n>m$: The only conformal polynomic parameterizations of a $m$-dimensional hypersurfaces, in  $\mathbb{R}^n$, with $m > 2$ and $n \geq m$, must be formed by lineal polynomials, i.e. the parameter must be a rotation, scale transformation, reflection or translation of the usual cartesian framework.\\
\end{abstract}

\newpage
\section{Conformal polynomial parameterizations}
\subsection{Harmonic and homogeneous polynomials}

\begin{defn}
A surface parameterization is conformal if satisfies the following condition:
\begin{equation}\label{cond_conformes_cartesianas}
\left\{
  \begin{array}{ll}
   &\bar{X}_x \cdot \bar{X}_x  - \bar{X}_y \cdot \bar{X}_y = 0 \\
   &\bar{X}_x \cdot \bar{X}_y = 0
  \end{array}
\right.
\end{equation}

In polar coordinates the condition is:
\begin{equation}\label{cond_conformes_polares}
\left\{
  \begin{array}{ll}
   &r^2 \bar{X}_r \cdot \bar{X}_r  - \bar{X}_{\theta} \cdot \bar{X}_{\theta} = 0 \\
   &r \bar{X}_r \cdot \bar{X}_{\theta} = 0
  \end{array}
\right.
\end{equation}
\end{defn}

\begin{defn}
A polynomial $p$ is said to be \emph{homogeneous} when all the monomial components has the same degree. In other words, $p$ is a homogeneous polynomial of degree $k$ in $\mathbb{R}^n$ if it has the following form:
\begin{equation*}
p(x_1,x_2, \ldots, x_n) = \sum_i a_i x_1^{\alpha_1^i} x_2^{\alpha_2^i} \cdots x_n^{\alpha_n^i}
\end{equation*}
where summation contains all combinations that satisfy $\sum_{j=1}^n \alpha_j^i = k \hspace{3mm} \forall i$.
\end{defn}
The set of $n$ variables \emph{homogeneous polynomials} will be denoted by $\mathcal{P}(\mathbb{R}^n)$ and the set of homogeneous polynomials of degree $i$ by $\mathcal{P}_i(\mathbb{R}^n)$.\\

\begin{defn}
A polynomial $p$ is said to be \emph{harmonic} when its laplacian is null, i.e. $\Delta p =0$.
\end{defn}
$\mathcal{H}(\mathbb{R}^n)$ will denote the \emph{harmonic homogeneous polynomials} set in $\mathbb{R}^n$ and $\mathcal{H}_i(\mathbb{R}^n)$ symbolize the harmonic homogeneous polynomial set of degree $i$.\\

The following known result, based on a more general one proved by Ernst Fischer in 1917 \cite{Fis}, will be used:\\
\begin{thm}\label{th_desc_homog}
Every homogeneous polynomial could be uniquely decomposed as a sum of harmonic homogeneous polynomials multiplied by $r^2$ powers.\\

More explicitly, every $m$ degree homogeneous polynomial $p \in \mathcal{P}_m(\mathbb{R}^n)$, could be decomposed as:
\begin{equation*}
p = h_m + r^2 h_{m-2} + \cdots + r^{2s} h_{m-2s}
\end{equation*}
where $s = [\frac{m}{2}]$ ([ ] is the integer part operator), and $r^2$ is the square of the position vector module in $\mathbb{R}^n$, $r^2 = x_1^2 + \cdots + x_n^2 = |x|^2$, and every $h_i$ are harmonic homogeneous polynomials of degree $i$, $h_i \in \mathcal{H}_i(\mathbb{R}^n)$.
\end{thm}
The proof of this theorem could be seen in \cite{Alx}, theorem 5.7.\\

In terms of polynomial spaces, the homogeneous polynomial space admits the decomposition in harmonic polynomial spaces, given by:
\begin{equation*}
\mathcal{P}_m(\mathbb{R}^n) = \mathcal{H}_m(\mathbb{R}^n) \oplus r^2 \mathcal{H}_{m-2}(\mathbb{R}^n) \oplus \cdots \oplus r^{2s}\mathcal{H}_{m-2s}(\mathbb{R}^n)
\end{equation*}
where $s = [\frac{m}{2}]$.\\

In the cited reference, \cite{Alx}, could be seen a proof of the following proposition about the harmonic polynomial space dimension:\\

\begin{prop}
If $m>2$ then:
\begin{equation}\label{dim_espacio_armonico}
dim \hspace{2mm} \mathcal{H}_m(\mathbb{R}^n) =
\left( {\begin{array}{cc}
 n+m-1  \\
 n-1  \\
 \end{array} } \right)
 -
\left( {\begin{array}{cc}
 n+m-3  \\
 n-1  \\
 \end{array} } \right)
\end{equation}
\end{prop}

In order to study conformal surface parameterizations only harmonic polynomials of two variables, $\mathcal{H}_m(\mathbb{R}^2)$ will be used. Following the above proposition, the harmonic polynomial base of any degree is always compound by two elements.\\

\begin{rem}\label{fourier}
The harmonic two variables $k$ degree polynomial space, $\mathcal{H}^k(\mathbb{R}^2)$, could be expressed as the real and imaginary part of $z^k$, where $z \in \mathbb{C}$. A base of the space $\mathcal{H}^k(\mathbb{R}^2)$ is $\{\hbox{Re}(z^k), \hbox{Im}(z^k)\}$. In other words, every $\mathcal{H}^k(\mathbb{R}^2)$ element is a lineal combination of  $\{z^m, \overline{z^m}\}$.\\

One harmonic homogeneous polynomial base in polar coordinates is:
\begin{equation}\label{base_fourier}
 \mathcal{H}^k(\mathbb{R}^2) = \{r^k \sin{k \theta}, r^k \cos{k \theta}\}
\end{equation}
this is the known Fourier base and the decomposition exposed in the theorem \ref{th_desc_homog}, applied in the two variable case, is the Fourier series expansion of any homogeneous polynomial.\\
\end{rem}

The two elements of the standard basis of $\mathcal{H}^m(\mathbb{R}^2)$ will be denoted by $\{h_1^m, h_2^m\}$.

\begin{rem}\label{notacion_vect}
The harmonic polynomial decomposition stated by theorem \ref{th_desc_homog} and a vector coefficient notation will be used. The next example tries to clarify the notation.\\

The standard basis of $\mathcal{H}^i(\mathbb{R}^2), \hspace{2mm} i=1,2,3$, are the next harmonic homogeneous polynomial pairs:
\begin{eqnarray*}
&& \mathcal{H}^1(\mathbb{R}^2) = \{h_1^1(x,y),h_2^1(x,y)\} = \{x,y\}\\
&& \mathcal{H}^2(\mathbb{R}^2) = \{h_1^2(x,y),h_2^2(x,y)\} =\{x^2 - y^2, 2xy\}\\
&& \mathcal{H}^3(\mathbb{R}^2) = \{h_1^3(x,y),h_2^3(x,y)\} = \{x^3 - 3xy^2, 3xy^2 - y^3\}
\end{eqnarray*}

The Enneper minimal surface has polynomic components and also is conformal. This surface could be expressed in terms of vector coefficients multiplied by harmonic base elements as:
\begin{eqnarray*}
&& \bar{\psi}(x,y) =  \big(x - x^3/3 + xy^2,  -y + y^3/3 - x^2 y,  x^2 - y^2 \big) = \\
&& = \bar{\lambda} h_1^3(x,y) + \bar{\beta} h_2^3(x,y) + \bar{\gamma} h_1^2(x,y) + \bar{\mu} h_1^1(x,y) + \bar{\nu} h_2^1(x,y)
\end{eqnarray*}
where:
\begin{equation*}
\bar{\lambda} =
\left( {\begin{array}{cc}
 -1/3  \\
 0  \\
 0  \\
 \end{array} } \right)
,\hspace{3mm}
\bar{\beta} =
\left( {\begin{array}{cc}
 0  \\
 -1/3  \\
 0  \\
 \end{array} } \right)
,\hspace{3mm}
\bar{\gamma} =
\left(
  \begin{array}{cc}
  0 \\
  0 \\
  1 \\
  \end{array}
\right)
,\hspace{3mm}
\bar{\mu} =
\left(
  \begin{array}{cc}
  1 \\
  0 \\
  0 \\
  \end{array}
\right)
,\hspace{3mm}
\bar{\nu} =
\left(
  \begin{array}{cc}
  0 \\
  -1 \\
  0 \\
  \end{array}
\right)
\end{equation*}
\end{rem}

The notation for the angular component of the $k$ degree Fourier basis elements will be:
\begin{equation*}
\bar{f}_k = \bar{v}_k \sin{k \theta} + \bar{•}_k \cos{k \theta}
\end{equation*}
where the radial factor, $r^k$, is deliberately eliminated.\\

The next definition will also be used:
\begin{equation*}
\bar{f}'_k = \bar{v}_k \cos{k \theta} - \bar{w}_k \sin{k \theta} = \frac1{k}\frac{d \bar{f}_k}{d\theta}
\end{equation*}

The next relations are consequences of the Fourier basis orthogonality properties in the unit circle $S_1$, ($r=1; \hspace{2mm} 0 \leq \theta \leq 2\pi$):
\begin{eqnarray}\label{orthog_prop}
&& \int_0^{2\pi}{|\bar{f}_k|^2} d\theta= \int_0^{2\pi}{|\bar{f}'_k|^2} d\theta= \pi(|\bar{v}_k|^2 + |\bar{w}_k|^2)\\
&& \int_0^{2\pi}{\bar{f}_k \cdot \bar{f}_i} d\theta = \int_0^{2\pi}{\bar{f}'_k \cdot \bar{f}'_i} d\theta = \pi (|\bar{v}_k|^2 + |\bar{w}_k|^2) \delta_{k i}\\
&& \int_0^{2\pi}{\bar{f}_k \cdot \bar{f}'_i} d\theta = 0\\
&& \int_0^{2\pi}{\bar{f}_k \cdot \bar{g}_k} d\theta = \int_0^{2\pi}{\bar{f}'_k \cdot \bar{g}'_k} d\theta = \pi (\bar{v}_k \cdot \bar{o}_k + \bar{w}_k \cdot \bar{q}_k)
\end{eqnarray}
where $\bar{o}_k$ and $\bar{q}_k$ are the vector coefficients of $\bar{g}_k$, the angular component of another harmonic polynomial.\\

\subsection{A conformal polynomial parameterization theorem}

\begin{thm}\label{th_param_conf_armonica}
Every conformal polynomial surface parameterization, embedded in $\mathbb{R}^n$, must be harmonic.\\

The general form of any conformal polynomial parameterization of $k$ degree in polar coordinates is:
\begin{equation*}
\bar{X} = \sum_{i=0}^k r^i (\bar{v}_i \sin{i \theta} + \bar{w}_i \cos{i \theta})
\end{equation*}
or using cartesian coordinates:
\begin{equation*}
\bar{X} = \sum_{i=0}^k \bar{v}_i \hbox{Re}(z^i) +  \bar{v}_i\hbox{Im}(z^i) = \sum_{i=0}^k \bar{v}_i h_1^i +  \bar{v}_i h_2^i
\end{equation*}
where $\{h_1^i, h_2^i\}$ are elements of a $\mathcal{H}_i(\mathbb{R}^2)$ base.\\

Also the vector coefficients of maximum, $j=k$, and minimum degree, $j=1$, must satisfy:
\begin{eqnarray*}
\left\{
  \begin{array}{ll}
   &|\bar{v}_j| = |\bar{w}_j| \\
   &\bar{v}_j \cdot \bar{w}_j = 0
  \end{array}
\right.
\end{eqnarray*}
\end{thm}

\begin{proof}[Proof]
The proof is decomposed in the following steps:
\begin{enumerate}
  \item Take a polynomial surface parameterization, in $\mathbb{R}^n$, of maximum degree $k$:
    \begin{equation*}
    \bar{X}(x,y) = (X_1(x,y), \ldots, X_n(x,y))
    \end{equation*}

  \item The polynomial parameterization splits into the sum of homogeneous components of degrees $1, \ldots, k$. The constant terms are neglected because the conformal condition is invariant under surface translations.\\
  
   Using the introduced vector notation, the polynomial surface parameterization takes the form:
    \begin{equation*}
    \bar{X} = \bar{P}_k + \bar{P}_{k-1} + \cdots + \bar{P}_1
    \end{equation*}
    where $\bar{P}_i$ are the vectors of $i$ degree homogeneous polynomials, $P_i \in \mathcal{P}_i(\mathbb{R}^2)$.\\

  \item The decomposition theorem \ref{th_desc_homog} is applied to each homogeneous component:
        \begin{eqnarray*}
        && \bar{P}_k = r^k\big(\bar{f}_k + \bar{g}_{k-2} + \cdots + \bar{h}_{k-2s_k}\big)\\
        && \bar{P}_{k-1} = r^{k-1}\big(\bar{f}_{k-1} + \bar{g}_{k-3} + \cdots + \bar{h}_{k-2s_{k-1}}\big)\\
        && \hspace{5mm} \vdots\\
        && \bar{P}_2 = r^2\big(\bar{f}_2 + \bar{g}_0\big)\\
        && \bar{P}_1 = r \bar{f}_1
        \end{eqnarray*}
where $s_k = [\frac{k}{2}]$ and vectors $\bar{f}_i,\bar{g}_i, \ldots, \bar{h}_i$ symbolize the angular component of the $i$ degree homogeneous polynomials and following the previous polar notation:
        \begin{eqnarray*}
        && \bar{f}_i = \bar{v}_i \sin{i \theta} + \bar{w}_i \cos{i \theta}\\
        && \bar{g}_i = \bar{o}_i \sin{i \theta} + \bar{q}_i \cos{i \theta}\\
        && \hspace{5mm} \vdots\\
        && \bar{h}_i = \bar{t}_i \sin{i \theta} + \bar{u}_i \cos{i \theta}
        \end{eqnarray*}

The $k$ degree polynomial parameterization takes the form:
    \begin{eqnarray*}
        && \bar{X}(r,\theta) = r^k\big(\bar{f}_k + \bar{g}_{k-2} + \cdots + \bar{h}_{k-2s_k}\big) + r^{k-1}\big(\bar{f}_{k-1} + \bar{g}_{k-3} + \cdots + \bar{h}_{k-2s_{k-1}}\big) + \\
        &&\hspace{3mm}+ \ldots + r^2\big(\bar{f}_2 + \bar{g}_0\big) + r \bar{f}_1
    \end{eqnarray*}
    
The maximum order harmonic terms, $\bar{f}_k, \bar{f}_{k-1}, \ldots $ will be called \emph{principal harmonic components} of the harmonic decomposition.\\   

  \item The tangent vectors are expressed in polar coordinates, the $r \bar{X}_r$ components are:
\begin{eqnarray*}
&& r \partial_r(\bar{P}_k) = k r^k[\bar{f}_k + \bar{g}_{k-2} + \cdots + \bar{h}_{k-2s_k}]\\
&& r \partial_r(\bar{P}_{k-1}) = (k-1) r^{k-1}[\bar{f}_{k-1} + \bar{g}_{k-3} + \cdots + \bar{h}_{k-2s_{k-1}}]\\
&& \hspace{5mm} \vdots\\
&& r \partial_r(\bar{P}_2) = 2r^2[\bar{f}_2 + \bar{g}_0]\\
&& r \partial_r(\bar{P}_1) = r \bar{f}_1
\end{eqnarray*}

and $\bar{X}_{\theta}$:
\begin{eqnarray*}
&& \partial_{\theta}(\bar{P}_k) = r^k[k\bar{f}'_k + (k-2)\bar{g}'_{k-2} + \cdots + (k-2s_k)\bar{h}'_{k-2s_k}]\\
&& \partial_{\theta}(\bar{P}_{k-1}) = r^{k-1}[(k-1)\bar{f}'_{k-1} + (k-3)\bar{g}'_{k-3} + \cdots + (k-2s_{k-1})\bar{h}'_{k-2s_{k-1}}]\\
&& \hspace{5mm} \vdots\\
&& \partial_{\theta}(\bar{P}_2) = r^2[2\bar{f}'_2]\\
&& \partial_{\theta}(\bar{P}_1) = r \bar{f}'_1
\end{eqnarray*}

This equalities are replaced on the first conformal parameterization condition (\ref{cond_conformes_polares}):
\begin{eqnarray*}
&& r^{2k}[k^2(\bar{f}_{k} \cdot \bar{f}_{k} - \bar{f}'_{k} \cdot \bar{f}'_{k}) + k^2 \bar{g}_{k-2} \cdot \bar{g}_{k-2} - (k-2)^2\bar{g}'_{k-2} \cdot \bar{g}'_{k-2} + \cdots +\\
&& \hspace{3mm}+ k^2 \bar{h}_{k-2s_k} \cdot \bar{h}_{k-2s_k} - (k-2s_k)^2\bar{h}'_{k-2s_k} \cdot \bar{h}'_{k-2s_k}] + \\
&& + r^{2k-2}[(k-1)^2(\bar{f}_{k-1} \cdot \bar{f}_{k-1} - \bar{f}'_{k-1} \cdot \bar{f}'_{k-1}) + (k-1)^2 \bar{g}_{k-3} \cdot \bar{g}_{k-3} - \\
&& \hspace{3mm}-(k-3)^2 \bar{g}'_{k-3} \cdot \bar{g}'_{k-3} + \cdots + (k-1)^2\bar{h}_{k-2s_{k-1}} \cdot \bar{h}_{k-2s_{k-1}} -\\
&& \hspace{3mm}-(k-2s_{k-1})^2\bar{h}'_{k-2s_{k-1}} \cdot \bar{h}'_{k-2s_{k-1}} + \ldots +\\
&& \hspace{3mm}+k(k-2)\bar{f}_{k} \cdot \bar{g}_{k-4} - (k-2)^2\bar{f}'_{k} \cdot \bar{g}'_{k-4} + \dots ] +\\
&& \hspace{5mm} \vdots\\
&& + r^6[3^2\bar{g}_1 \cdot \bar{g}_1 - \bar{g}'_1 \cdot \bar{g}'_1 + (4 \cdot 2) \bar{g}_2 \cdot \bar{f}_2 - (2 \cdot 2) \bar{g}'_2 \cdot \bar{f}'_2 + 5 \bar{h}_1 \cdot \bar{f}_1 - \bar{h}'_1 \cdot \bar{f}'_1] + \\
&& + r^4[2^2 \bar{g}_0 \cdot \bar{g}_0 - \bar{g}'_0 \cdot \bar{g}'_0 + 3\bar{g}_1 \cdot \bar{f}_1 - \bar{g}'_1 \cdot \bar{f}'_1] + r^2[\bar{f}_{1} \cdot \bar{f}_{1} - \bar{f}'_{1} \cdot \bar{f}'_{1}]\\
&& + \sum_{i=2}^{k}r^{2i-1}[\ldots] = 0
\end{eqnarray*}
The last term groups odd $r$ powers because it will be null when the equation is integrated on the unit circle $S_1$, by the harmonic polynomial orthogonality properties (\ref{orthog_prop}).\\

  \item Coefficients of different $r$ powers must be null simultaneously, because conformal condition must be satified in all the space. The following relations are obtained:
\begin{eqnarray*}
&& k^2(\bar{f}_{k} \cdot \bar{f}_{k} - \bar{f}'_{k} \cdot \bar{f}'_{k}) + k^2 \bar{g}_{k-2} \cdot \bar{g}_{k-2} - (k-2)^2\bar{g}'_{k-2} \cdot \bar{g}'_{k-2} + \cdots +\\
&& \hspace{3mm}+ k^2 \bar{h}_{k-2s_k} \cdot \bar{h}_{k-2s_k} - (k-2s_k)^2\bar{h}'_{k-2s_k} \cdot \bar{h}'_{k-2s_k} = 0\\
&&\\
&& (k-1)^2(\bar{f}_{k-1} \cdot \bar{f}_{k-1} - \bar{f}'_{k-1} \cdot \bar{f}'_{k-1}) + \\
&&\hspace{5mm} + (k-1)^2 \bar{g}_{k-3} \cdot \bar{g}_{k-3} - (k-3)^2 \bar{g}'_{k-3} \cdot \bar{g}'_{k-3} + \cdots \\
&& \cdots + (k-1)^2\bar{h}_{k-2s_{k-1}} \cdot \bar{h}_{k-2s_{k-1}} -(k-2s_{k-1})^2\bar{h}'_{k-2s_{k-1}} \cdot \bar{h}'_{k-2s_{k-1}} + \ldots \\
&& \ldots + k(k-2)\bar{f}_{k} \cdot \bar{g}_{k-4} - (k-2)^2\bar{f}'_{k} \cdot \bar{g}'_{k-4} + \dots =0\\
&& \hspace{5mm} \vdots\\
&& 3^2\bar{g}_1 \cdot \bar{g}_1 - \bar{g}'_1 \cdot \bar{g}'_1 + (4 \cdot 2) \bar{g}_2 \cdot \bar{f}_2 - (2 \cdot 2) \bar{g}'_2 \cdot \bar{f}'_2 + 5 \bar{h}_1 \cdot \bar{f}_1 - \bar{h}'_1 \cdot \bar{f}'_1 = 0 \\
&&\\
&& 2^2 \bar{g}_0 \cdot \bar{g}_0 - \bar{g}'_0 \cdot \bar{g}'_0 + 3\bar{g}_1 \cdot \bar{f}_1 - \bar{g}'_1 \cdot \bar{f}'_1 = 0 \\
&&\\
&& \bar{f}_{1} \cdot \bar{f}_{1} - \bar{f}'_{1} \cdot \bar{f}'_{1}= 0 \\
&&\\
&& [\text{Odd $r$ powers}] = 0
\end{eqnarray*}

  \item Now each equation is integrated on the unit circle.\\
   Using the cited Fourier basis orthogonality properties, (\ref{orthog_prop}), the last term, corresponding to products of principal harmonic components, have identical coefficients $k^2, (k-1)^2, \ldots$:
    \begin{equation*}
       \int_0^{2\pi} \bar{f}_{k} \cdot \bar{f}_{k} - \bar{f}'_{k} \cdot \bar{f}'_{k} d\theta = \int_0^{2\pi} \bar{f}_{k-1} \cdot \bar{f}_{k-1} - \bar{f}'_{k-1} \cdot \bar{f}'_{k-1} d\theta = 0
    \end{equation*}

   Harmonic component products of different degree are null (it includes all the odd powers of $r$) because of the Fourier basis orthogonality properties (\ref{orthog_prop}):
    \begin{equation*}
       \int_0^{2\pi} \bar{f}_{m} \cdot \bar{f}_{n} d\theta = \int_0^{2\pi} \bar{f}'_{m} \cdot \bar{f}'_{n} d\theta = 0
    \end{equation*}
    with $m \neq n$.\\

The equations could be simplified and take the form:
\begin{eqnarray*}
&& (k^2 - (k-2)^2) \int_0^{2\pi} \bar{g}_{k-2} \cdot \bar{g}_{k-2} d\theta + \cdots + \\
&& \hspace{3mm}+(k^2 - (k-2s_k)^2)\int_0^{2\pi} \bar{h}_{k-2s_k} \cdot \bar{h}_{k-2s_k}d\theta = 0 \\
&&\\
&& ((k-1)^2 - (k-3)^2)\int_0^{2\pi} \bar{g}_{k-3} \cdot \bar{g}_{k-3} d\theta + \cdots + \\
&& \hspace{3mm}+((k-1)^2 - (k-2s_{k-1})^2)\int_0^{2\pi} \bar{h}_{k-2s_{k-1}} \cdot \bar{h}_{k-2s_{k-1}}d\theta + \ldots +\\
&& \hspace{3mm}+(k (k-2) - (k-2)^2)\int_0^{2\pi}\bar{g}_{k-2} \cdot \bar{f}_{k-2}d\theta + \dots = 0\\
&& \hspace{5mm} \vdots \\
&& 8 \int_0^{2\pi} \bar{g}_1 \cdot \bar{g}_1 d\theta + 4 \int_0^{2\pi}\bar{g}_2 \cdot \bar{f}_2 d\theta + 4 \int_0^{2\pi}\bar{h}_1 \cdot \bar{f}_1 d\theta = 0\\
&&\\
&& 4 \int_0^{2\pi} \bar{g}_0 \cdot \bar{g}_0 d\theta + 2 \int_0^{2\pi} \bar{g}_1 \cdot \bar{f}_1 d\theta = 0
\end{eqnarray*}

   There are only two types of products of the same degree:
	\begin{itemize}
	  \item Terms corresponding to square powers of harmonic components, like:
        \begin{equation*}
            \int_0^{2\pi} \bar{g}_{k-1} \cdot \bar{g}_{k-1} d\theta = \pi (|\bar{o}_k|^2 + |\bar{q}_k|^2)
        \end{equation*}
	  \item Cross-products of different degree harmonic components, for example:
        \begin{equation*}
            \int_0^{2\pi} \bar{g}_{k} \cdot \bar{f}_{k-2} d\theta
        \end{equation*}
	\end{itemize}
Using the orthogonality properties (\ref{orthog_prop}), the above equations are simplified into:
    \begin{eqnarray*}
    && \pi [(k^2 - (k-2)^2) (|\bar{o}_{k-2}|^2 + |\bar{q}_{k-2}|^2) + \cdots + \\
    && \hspace{3mm}+(k^2 - (k-2s_k)^2) (|\bar{t}_{k-2s_k}|^2 + |\bar{u}_{k-2s_k}|^2)] = 0 \\
    &&\\
    && ((k-1)^2 - (k-3)^2) \pi (|\bar{o}_{k-3}|^2 + |\bar{q}_{k-3}|^2) + \cdots + \\
    && \hspace{3mm}+((k-1)^2 - (k-2s_{k-1})^2) \pi (|\bar{t}_{k-2s_{k-1}}|^2 + |\bar{u}_{k-2s_{k-1}}|^2) + \ldots +\\
    && \hspace{3mm}+(k (k-2) - (k-2)^2)\int_0^{2\pi}\bar{g}_{k-2} \cdot \bar{f}_{k-2}d\theta + \dots = 0\\
    && \hspace{5mm} \vdots \\
    && 8 \pi (|\bar{o}_1|^2 + |\bar{q}_1|^2) + 4 \int_0^{2\pi}\bar{g}_2 \cdot \bar{f}_2 d\theta + 4 \int_0^{2\pi}\bar{h}_1 \cdot \bar{f}_1 d\theta = 0\\
    &&\\
    && 4 \int_0^{2\pi} \bar{g}_0 \cdot \bar{g}_0 d\theta + 2 \int_0^{2\pi} \bar{g}_1 \cdot \bar{f}_1 d\theta = 0
    \end{eqnarray*}

  \item The second type of terms, the cross-products, are removed gradually. Each equation, that represents the coefficient of a different $r$ power, make null the square of the terms that appears on the cross-terms of the next $r$ power.\\

      On the first equation, the square elements must be null, because all the coefficients are positive and could be deduced:
      \begin{equation*}
      \hspace{5mm} \bar{o}_{k-2} = \bar{q}_{k-2} = 0  \hspace{5mm} \text{o} \hspace{5mm} |\bar{g}_{k-2}| = 0  \Rightarrow \bar{g}_{k-2} = 0
      \end{equation*}

      In the second equation the cross-product terms vanish because it contains the product of the harmonic components $g_{k-2}$:
      \begin{equation*}
      \int_0^{2\pi}\bar{g}_{k-2} \cdot \bar{f}_{k-2}d\theta = 0
      \end{equation*}

      If we remove this cross term, the new equation cancel the harmonic coefficients of the next lower degree:
      \begin{equation*}
      \bar{o}_{k-3} = \bar{q}_{k-3} = \cdots = \bar{t}_{k-2s_{k-1}} = \bar{u}_{k-2s_{k-1}} = 0
      \end{equation*}

  \item This procedure is iterated over all the equations, the only not null resulting terms are the square terms of the principal harmonic components. The coefficients of the  squares of non principal harmonic components are positive (it takes the form $k^2 - (k-i)^2$, with $i<k$ and $k>1$) so all the quadratic terms must be null simultaneously:
       $|\bar{g}_i| = \ldots = |\bar{h}_j| = 0$.\\
      \begin{eqnarray*}
        && \int_{0}^{2\pi}\bar{g}_i \cdot \bar{g}_i d\theta = \cdots =\int_{0}^{2\pi}\bar{h}_j \cdot \bar{h}_j d\theta = 0 \Rightarrow\\
        && |\bar{o}_i| = |\bar{q}_i| = \cdots = |\bar{t}_i| = |\bar{u}_i| = 0 \Rightarrow\\
        && \bar{g}_i = \cdots \bar{h}_j = 0
      \end{eqnarray*}
       Therefore, all non principal harmonic components are null. The polynomial parameterization could only contain principal harmonic components. In other words, the function components of the polynomial parameterization must be harmonic.

  \item Now could be applied the conformal parameterization condition (\ref{cond_conformes_polares}) to the harmonic polynomial parameterization and then $r$ powers could be grouped obtaining additional conditions for vector coefficients. The more simple conditions, for the higher, $j=k$, and lower vector coefficients, $j=1$, are:
    \begin{equation*}
    \left\{
      \begin{array}{ll}
       &|\bar{v}_j| = |\bar{w}_j|\\
       &\bar{v}_j \cdot \bar{w}_j = 0
      \end{array}
    \right.
    \end{equation*}
\end{enumerate}
\end{proof}

\begin{rem}
The above proof only uses the first conformal parameterization condition (\ref{cond_conformes_polares}):
\begin{equation*}
\bar{X}_x \cdot \bar{X}_x - \bar{X}_y \cdot \bar{X}_y = 0
\end{equation*}

It could be thought that the second condition:
\begin{equation*}
\bar{X}_x \cdot \bar{X}_y = 0
\end{equation*}
imposes additional restrictions. In general, for non polynomial parameterizations, that is true. In the polynomial case it will be shown that this condition is superfluous.\\

The conformal polynomial parameterization could be written on the complex plane as:
\begin{equation*}
\bar{X}(r,\theta) = \sum_{j=1}^n r^j (\bar{v}_j \cos{j \theta} + \bar{w}_j \sin{j \theta}) = \sum_{j=1}^n \bar{V}_j z^j
\end{equation*}
con $z, \bar{V}_j \in \mathbb{C}$ y $\bar{V}_j = \bar{v}_j - \mathbf{i}\bar{w}_j$.\\

The surface tangent vectors are given by:
\begin{eqnarray*}
&& \bar{X}_x = \bar{X}_z + \bar{X}_{\bar{z}}\\
&& \bar{X}_y = i(\bar{X}_z - \bar{X}_{\bar{z}})
\end{eqnarray*}

The first condition on the complex plane is:
\begin{eqnarray*}
&& \bar{X}_x \cdot \bar{X}_x - \bar{X}_y \cdot \bar{X}_y = 0 \Rightarrow\\
&& \Rightarrow \bar{X}_z \cdot \bar{X}_z + \bar{X}_{\bar{z}} \cdot \bar{X}_{\bar{z}} = 0
\end{eqnarray*}
This condition takes the form:
\begin{equation*}
\bar{X}_z \cdot \bar{X}_z = 0
\end{equation*}

The second conformal parameterization condition on the complex plane is:
\begin{eqnarray*}
&& \bar{X}_x \cdot \bar{X}_y = 0 \Rightarrow\\
&& i(\bar{X}_z \cdot \bar{X}_z - \bar{X}_{\bar{z}} \cdot \bar{X}_{\bar{z}}) = 0\\
&& i(\bar{X}_z \cdot \bar{X}_z) = 0
\end{eqnarray*}
If the parameterization $X$ is harmonic and polynomial, it could be expressed as a polynomial in $z$ variable, see remark \ref{fourier}, i.e., it must be an holomorphic function. 
When $\bar{z}$ derivatives are neglected it could be seen that the two conditions are equivalent.
\end{rem}

\newpage
\section{Geometric applications}
\subsection{Minimal surfaces}
We use the next notation for the elements of the metric or the \emph{first fundamental form}:
\begin{equation*}
\left\{
\begin{array}{ll}
E = \bar{X}_x \cdot \bar{X}_x \\
F = \bar{X}_x \cdot \bar{X}_y =  \bar{X}_y \cdot \bar{X}_x\\
G = \bar{X}_y \cdot \bar{X}_y
\end{array} \right.
\end{equation*}

Using this notation, a surface parameterization is conformal when:
\begin{eqnarray*}
\left\{
\begin{array}{ll}
E = G \\
F = 0
\end{array}
\right.
\end{eqnarray*}

The elements of the \emph{second fundamental form} are:
\begin{equation*}
\left\{
\begin{array}{ll}
e =  \bar{X}_{xx} \cdot (\bar{X}_x \wedge \bar{X}_y)\\
f =  \bar{X}_{xy} \cdot (\bar{X}_x \wedge \bar{X}_y)\\
g =  \bar{X}_{yy} \cdot (\bar{X}_x \wedge \bar{X}_y)
\end{array} \right.
\end{equation*}

The \emph{mean curvature} expression, see for example \cite{Docarmo}, as a function of the first and second fundamental forms is:
\begin{equation}\label{COND_H_NULA1}
H = \frac{1}{2}\frac{Eg - 2fF + Ge}{EG - F^2}
\end{equation}

The following geometric result is obtained from the theorem \ref{th_param_conf_armonica} of the previous section.\\

\begin{cor}\label{corolario_sup_min}
Every Riemannian surface $M$ in $\mathbb{R}^n$ that admits a conformal polynomial parameterization must be a minimal surface.
\end{cor}

\begin{proof}[Proof]
Let $\bar{X}$ be a conformal parameterization. The theorem \ref{th_param_conf_armonica} states that every conformal polynomial parameterization must be also harmonic, $\Delta \bar{X}=0$

In the conformal parameterization case, because $E = G, F = 0$, the mean curvature is given by:
\begin{eqnarray}\label{COND_H_NULA2}
&&H = \frac{1}{2}\frac{Eg + Ge}{EG} = \frac{1}{2}\frac{g + e}{E}= \\
&&= \frac{1}{2}\frac{(\bar{X}_{uu} + \bar{X}_{vv}) \cdot (\bar{X}_x \wedge \bar{X}_y)}{\bar{X}_x \cdot \bar{X}_x} = \frac1{2} \Delta \bar{X} \cdot \frac{\bar{N}}{E}
\end{eqnarray}

The fact that the parameterization has to be harmonic, in all of his components, implies that surface mean curvature must be null, i.e. the surface must be minimal.
\end{proof}

\begin{rem}
The use of conformal parameters is very common in the resolution of physical and engineering problems because many describing differential equations are simplified making use of this kind of surface coordinates (the local existence of conformal coordinates is guaranteed as can be seen in \cite{Docarmo} or \cite{Duv}). In addition to this, many times the solutions are approximated by a Taylor polynomial, or any other polynomial series expansion, around a point.\\

Using the above result, this kind of aproximmations are really minimal surface approximations.\\

In the next section it will be seen that this condition is still more restrictive for conformal polynomic hypersurfaces.
\end{rem}

\subsection{Conformal spinorial surface representation}
The Weierstrass-Enneper surface representation is used to generate minimal surface conformal parameter based on two complex functions.\\

This idea has been extended to obtain corformal parameters of any kind of surface, not only the minimal ones, in $\mathbb{R}^3$ and $\mathbb{R}^4$. See \cite{Tai}, \cite{Kam},  \cite{Kus}, \cite{KKP}, \cite{Sul}, \cite{Kam2} for a detailed description.\\

The spinorial differential equations used to obtain the components of the conformal parameter, see \cite{Tai}, are:
\begin{equation*}
\mathcal{D}\psi = 0
\end{equation*}
where $\mathcal{D}$ corresponds to the complex bidimensional Dirac operator and $A$ is a  real scalar potential:
\begin{equation*}
\mathcal{D} =
\left( {\begin{array}{cc}
0 & \partial_z \\
-\partial_{\bar{z}} & 0 \\
 \end{array} } \right)
 +
\left( {\begin{array}{cc}
                A & 0 \\
                0 & A \\
 \end{array} } \right)
\end{equation*}

As noted by \cite{Tai} this equations corresponds to the stationary Dirac equation in presence of a external scalar electromagnetic field. The theorem \ref{th_param_conf_armonica} implies that  there are no polynomial solution for the above equation with non null potentials. In other words, there are non polynomial spinors (the components of the stationary wave function) when potential $A$ is not null.\\

Homogeneous system, $A=0$, corresponds to spinorial representation of minimal surfaces.

\subsection{A rational conformal counterexample}
The theorem \ref{th_param_conf_armonica} establish that every conformal polynomial parameterization of a embedded surface in $\mathbb{R}^n$ must be harmonic and, by the corollary \ref{corolario_sup_min}, it also must be a minimal surface.\\

That is not the case for conformal rational polynomial parameterizations, where the components are quotients of polynomials. One counterexample could be found using a conformal transformation of a known conformal polynomial parameterization, for example the Enneper minimal surface.\\

The \emph{conformal transformation group} is the transformation group that conformally changes $\mathbb{R}^n$. For $n>2$, as the Liouville theorem states, this group is composed by translations, scale transformations, rotations and special conformal group transformations, $SCG(\mathbb{R}^n)$.\\

The special conformal group, $SCG(\mathbb{R}^n)$, is a subgroup which elements could be expressed as the composition of a inversion, $R$, a translation by a vector $\bar{a}$, called $T(\bar{a})$ and a new radix inversion $R$. In other words, for every $S \in SCG(\mathbb{R}^n)$ exist a vector $\bar{a}$ that satisfies $S = R \cdot T(\bar{a}) \cdot R$.\\

This kind of conformal transformation could be applied to a conformal parameterization of the Enneper minimal surface to obtain a new conformal parameterization of a different surface in $\mathbb{R}^n$. The new conformal surface is not necessarily minimal although the original surface is minimal because the mean curvature $H$ is not a conformal invariant (not as the Willmore integrand, $(H^2 - K)d\sigma$, where $d\sigma$ is the area differential, that is conformally invariant).\\

A lot of counterexamples of conformal rational polynomial surface parameterizations could be obtained with not everywhere null curvature, i.e., the new surfaces are not minimal and also not harmonic. There are also non Willmore surface examples obtained using interesting spinorial technics that will not be included here.

\newpage
\section{A general theorem for hypersurfaces}
The idea of the theorem \ref{th_param_conf_armonica} could be generalized to $m$-dimensional hypersurfaces in $\mathbb{R}^n$. It will be seen that there are rigidity conditions, as restrictive as the established by the Liouville theorem.\\

The classical Liouville theorem states that every conformal transformation of a space region in $\mathbb{R}^n$, with $n > 2$, could be expressed as a composition of some of the following operations: inversions, translations, rotations y scale transformations. For a proof see for example \cite{Bla} or \cite{Spi}, vol 3.\\

In fact, it will be shown that the only conformal polynomial parameterizations of a hypersurface must be composed lineal polynomials. In other words, every conformal polynomial parameterization of a $m$-dimensional hypersurface, embedded in $\mathbb{R}^n$, is essentially a rotation, translation or scale transformation of the cartesian framework.

\begin{thm}
Every conformal polynomic parameterization of a $m$-dimensional surface, with $m > 2$, embedded in $\mathbb{R}^n$, must be lineal, i.e., it must be a hyperplane.\\

In other words, the surface parameterization must be a lineal conformal transformation (rotations, translations or scale transformations) of the $m$-dimensional cartesian framework.
\end{thm}

\begin{proof}[Proof]
Let $\bar{\psi}$ be a conformal polynomial parameterization of a $m$-dimensional surface embedded in $\mathbb{R}^n$ and let $\bar{\phi}$ be a conformal polynomial parameterization of a bidimensional surface embedded in $\mathbb{R}^n$. The two conformal polynomial parameterizations are:
\begin{equation*}
\left\{
  \begin{array}{ll}
   &\bar{\phi}(x, y):\mathbb{R}^2 \rightarrow \mathbb{R}^m \\
   &\bar{\psi}(x_1, \ldots,x_m):\mathbb{R}^m \rightarrow \mathbb{R}^n
  \end{array}
\right.
\end{equation*}

The composition of both maps, $\bar{X}: \mathbb{R}^2 \rightarrow \mathbb{R}^n$, must be a conformal map too and it is a conformal polynomial parameterization of a bidimensional surface in $\mathbb{R}^n$. The theorem \ref{th_param_conf_armonica} establish that this parameterization must also be harmonic. Thus every component, $X_i$, of the parameterization:
\begin{equation*}
 X^i = \psi^i(\phi^1(x, y), \phi^2(x, y), \ldots, \phi^m(x, y))
\end{equation*}
must be harmonic, $\Delta X^i = 0$.\\

The laplacian of each component could be calculated explicitly:
\begin{eqnarray}
&& X^i_x = \sum_{j=1}^m \psi^i_j\bigg|_{\bar{\phi}} \phi^j_x \\
&& X^i_{xx} = \sum_{j,k=1}^m  \psi^i_{jk}\bigg|_{\bar{\phi}} \phi^j_x \phi^k_x + \sum_{j=1}^m \psi^i_j\bigg|_{\bar{\psi}} \phi^j_{xx}\\
&& X^i_{yy} = \sum_{j,k=1}^m  \psi^i_{jk}\bigg|_{\bar{\phi}} \phi^j_y \phi^k_y + \sum_{j=1}^m \psi^i_j\bigg|_{\bar{\psi}} \phi^j_{yy}\\
&& \Delta X^i = \sum_{j,k=1}^m  \psi^i_{jk}\bigg|_{\bar{\phi}} \phi^j_x \phi^k_x + \sum_{j,k=1}^m  \psi^i_{jk}\bigg|_{\bar{\phi}} \phi^j_y \phi^k_y + \sum_{j=1}^m \psi^i_j\bigg|_{\bar{\psi}} \Delta \phi^j = \\
&& \Delta X^i = \sum_{j,k=1}^m  \psi^i_{jk}\bigg|_{\bar{\phi}} \phi^j_x \phi^k_x + \sum_{j,k=1}^m  \psi^i_{jk}\bigg|_{\bar{\phi}} \phi^j_y \phi^k_y
\end{eqnarray}
The term $\Delta \phi^j$ could be removed because the components are conformal and polynomic and has to be harmonic.\\

Finally, the laplacian of the parameter components $X^i$ is:
\begin{equation}\label{hess}
\Delta X^i = \bar{\phi_x} \cdot Hess(\psi^i) \cdot \bar{\phi_x} + \bar{\phi_y} \cdot Hess(\psi^i) \cdot \bar{\phi_y}
\end{equation}
where $Hess(\psi^i)$ is the hessian matrix of the $\psi$ parameterization $i$ component.\\

The previous relation must be true for every conformal parameterization $\bar{\phi}$, of a surface in $\mathbb{R}^m$, used on the composition $\bar{X} = \bar{\psi}(\bar{\phi}(x,y))$. The simpler polynomial parameterization is a lineal one:
\begin{equation*}
\bar{\phi}(x,y) = \bar{\lambda}x +\bar{\beta}y
\end{equation*}
In order to be conformal it must satisfy:
\begin{equation*}
\left\{
  \begin{array}{ll}
   &|\bar{\lambda}| = |\bar{\beta}| \\
   &\bar{\lambda} \cdot \bar{\beta} = 0
  \end{array}
\right.
\end{equation*}
The relation (\ref{hess}) must be satisfied at every point $p \in \mathbb{R}^m$:
\begin{equation}\label{rel_alg_lineal}
\Delta X^i\big|_p = \bar{\lambda} \cdot A \cdot \bar{\lambda} + \bar{\beta} \cdot A \cdot \bar{\beta} = 0
\end{equation}
with $A \equiv Hess(\psi^i)\big|_p$, where $A$ is a symmetric matrix because it symbolize the Hessian matrix evaluated at the point $p$. The later equation is equivalent to the projection of a bilinear form represented by the symmetric matrix $A$ into the plane formed by the vectors $\bar{\lambda}, \bar{\beta}$. The relation must be true for every conformal parameterization $\phi$, the pair of vectors $\{\bar{\lambda},\bar{\beta}\}$ could be chosen to be $A$ matrix eigenvectors, $\{\bar{v_1}, \bar{v_2}\}$.\\

The only requirement for the vectors $\bar{\lambda}, \bar{\beta}$ is that they have to be orthogonal and of identical size. For examble, it could be taken unitary. This condition is fulfilled by the  matrix $A$ eigenvectors because $A$ is real and symmetric. The spectral theorem for finite spaces states that every real symmetric matrix could be diagonalized in a orthogonal basis and the eigenvalues must be real numbers.\\

Using the eigenvectors, the equation (\ref{rel_alg_lineal}) gives:
\begin{equation}
\Delta X^i\bigg|_p = \bar{v_1} \cdot A \cdot \bar{v_1} + \bar{v_2} \cdot A \cdot \bar{v_2} = (\lambda_1 + \lambda_2)|\bar{v_1}| = 0
\end{equation}
where $\lambda_i$ is the eigenvalue associated to the eigenvector $v_i$.\\

The same reasoning could be applied to the other $m$ eigenvectors of the $A$ matrix:
\begin{equation*}
\lambda_i + \lambda_j = 0 \hspace{5mm} \forall i \neq j  \hspace{5mm} i,j = 1, \ldots, m\\
\end{equation*}

This lineal equation system has only null solution for $m > 2$. In other words, the matrix $A$ must be null, or equivalently, $Hess(\psi^i) = 0$ at every point $p$. All the second order derivatives and the second order cross derivatives of the parameterization components must be null and thus the conformal polynomial parameterization must be lineal.\\
\end{proof}

\begin{rem}
The above result is coherent with the Liouville theorem, when the dimension values $m,n$ are the same. In the case $n = m$, the only conformal polynomial transformations, from $\mathbb{R}^m$ to $\mathbb{R}^m$, allowed by the Liouville theorem are the lineal ones. This lineal transformations corresponds to the composition of rotations, scale transformations and translations. The special conformal subgroup transformations could not be used because it contains radix inversions that are not polynomial transformations.\\
\end{rem}

\newpage
\bibliographystyle{amsplain}

\end{document}